\documentclass[11pt,reqno]{amsart}
\usepackage{amsthm,amssymb,amsfonts}
\usepackage[nosumlimits]{mathtools}

\usepackage{graphicx}
\usepackage{subcaption}
\usepackage{algorithm}
\usepackage{algpseudocode}
\usepackage{tikz-cd}
\usepackage{enumerate}
\usepackage[margin=1in]{geometry}
\usepackage{booktabs}
\usepackage{url} 

\usepackage{hyperref,xcolor}
\hypersetup{
    colorlinks,
    linkcolor={red!50!black},
    citecolor={blue!50!black},
    urlcolor={blue!80!black}
}

\theoremstyle{plain}
\newtheorem{theorem}{Theorem}[section]
\newtheorem{proposition}[theorem]{Proposition}
\newtheorem{lemma}[theorem]{Lemma}
\newtheorem{corollary}[theorem]{Corollary}

\theoremstyle{definition}

\renewcommand{\algorithmicrequire}{\textbf{Input:}}
\renewcommand{\algorithmicensure}{\textbf{Output:}}

\DeclareMathOperator{\GL}{GL}

\DeclareMathOperator{\im}{Im}

\begin{document}
\title{Generalizing Frobenius Inversion to Quaternion Matrices}

\author[Q-Y.~Chen]{Qiyuan~Chen}
\address{KLMM, Academy of Mathematics and Systems Science, Chinese Academy of Sciences, Beijing 100190, China}
\email{chenqiyuan@amss.ac.cn}
\author[J.~Uhlmann]{J.\,Uhlmann}
\address{Department of Electrical Engineering and Computer Science, University of Missouri - Columbia}
\email{uhlmannj@missouri.edu}
\author[K.~Ye]{Ke Ye}
\address{KLMM, Academy of Mathematics and Systems Science, Chinese Academy of Sciences, Beijing 100190, China}
\email{keyk@amss.ac.cn}
\date{}

\begin{abstract}
In this paper we derive and analyze an algorithm for inverting quaternion
matrices. The algorithm is an analogue of the Frobenius algorithm for the complex matrix inversion. On the theory side, we prove that our algorithm is more efficient that other existing methods. Moreover, our algorithm is optimal in the sense of the least number of complex inversions. On the practice side, our algorithm outperforms existing algorithms on randomly generated matrices. We argue that this algorithm can be used to improve the practical utility of 
recursive Strassen-type algorithms by providing the fastest possible base 
case for the recursive decomposition process when applied to quaternion 
matrices.
\end{abstract}

\maketitle

\section{Introduction}
Quaternions are a typical example of hypercomplex number systems, which was first discovered by Hamilton in 1843 \cite{Hamilton10}. Since then, the quaternions have been extensively studied and applied in various fields of mathematics including contact and spin geoemtry \cite{Arnold1995,HM89}, function theory \cite{Sudbery79,Deavours73}, non-commutative algebra \cite{Hurwitz1922,Chevalley1996,Cohn1995}, linear algebra \cite{Fuzhen1997} and algebraic topology \cite{Hopf1931}. On the other side, since the group of unit quaternions is isomorphic to the group $\operatorname{SU}(2)$ consisting of $2\times 2$ special unitary matrices, they provide us a natural way to represent spatial rotations.\footnote{$\operatorname{SU}(2)$ is a double cover of the group $\operatorname{SO}(3)$ of all spatial rotations.} Because of such a surprising connection, quaternions are widely applied in areas related to rigid body mechanics and quantum mechanics \cite{Berthold87,GRUBIN70, Arribas2006,Faugeras86,Paldus2013,NSN17,Shoemake85}. Quaternions are also useful for combining relating variables into a single algebraic entity that can allow them to be manipulated collectively in a natural way, e.g., representing RGB pixel values using the $i$, $j$, and $k$ components of a single quaternion. During the past decade, novel quaternion-based methods and models have emerged in a wide variety of signal processing applications \cite{VRS10, Ghouti2014,YLHHT15,VPVS11}.

The increasing prevalence of quaternions in science and engineering has led to increased focus on the development of high-efficiency implementations of core utilities such as the quaternion LU decomposition \cite{WLZZ2018}, the quaternion SVD \cite{JNS19}, low rank approximation of quaternion matrices \cite{LLJ22}, and high performance quaternion matrix multiplication \cite{Williams2019}.

The focus of this paper is on arguably the most fundamental quaternion matrix operation after multiplication: {\em quaternion matrix inversion}.  Presently there are only four prominent algorithms for matrix inversion over quaternions in the literature \cite{WLZZ2018,Fuzhen1997,Turner2009,SB2022}, and their use includes applications such as computing or estimating the inverse of a covariance matrix in statistics and machine learning \cite{Anderson58,VPVS11,Yuan10}; quaternion batch normalization \cite{CA18,TMG19}; construction of orthogonal polynomials \cite{DN59}; the computation of numerical integrals \cite{MGJ91}; and hardware design of MIMO radios \cite{CSD11}.

The structure of this paper is as follows. We begin with background
on the algorithms of matrix inversion and how the most efficient
algorithm depends on the type (e.g., real, complex, or quaternion)
of the matrix. We then derive an algorithm that is optimized for 
quaternion matrix inversion. After that, we discuss the computational complexity of our algorithm from two different points of view. Finally, we provide empirical results 
demonstrating its efficiency advantage over prior methods 
presented in the literature.

\section{Background}

The development of practically efficient methods for manually inverting matrices
dates back to the early 20th Century. An important early result due to Frobenius
in 1910 was a method for decomposing the inverse of a complex matrix into
a form requiring only the evaluation of inverses of real matrices. Interest in 
methods of this type was renewed in the 1950s and 1960s as a means for
inverting complex matrices in contexts in which complex variables were not
natively supported, e.g., when numerical libraries contained only optimized
routines for real matrices.

In 1969, Strassen showed that a special recursive block decomposition of 
the matrix multiplication problem could achieve better complexity
than the conventional $O(n^3)$-time algorithm. Moreover, he showed that
this implies related recursive algorithms with the same subcubic complexity 
for computing the determinant and inverse of a matrix. Refinements of the 
approach have led to progressively improving computational complexities  
over the decades \cite{CoppersmithW90}, with $O(n^{2.37286})$ being
the current best complexity for matrix multiplication \cite{JV21} and, 
consequently, matrix inversion.

While Strassen-type algorithms provide better {\em asymptotic}
computational complexities, their practical utility depends on the
absolute efficiency offered for the problem at hand, i.e., whether
it is faster than the best cubic-complexity algorithm for the matrix size 
of interest. More specifically, the practical efficiency of a generic 
Strassen-type recursion depends on the base-case efficiency of 
the fastest available cubic-complexity algorithm. In other words, 
there exists a fixed breakeven matrix size at which it is faster 
to apply a cubic-complexity algorithm than it is to continue the recursive 
block decomposition process down to the scalar base case. In the case of
real matrix multiplication, the best base-case algorithm is the conventional 
cubic-complexity matrix multiplication algorithm. In the case of 
matrix inversion, however, the situation is much more subtle.

In \cite{DLY22}, the authors established a surprising result that the 
Frobenius algorithm for inverting a complex matrix incurs less numerical 
overhead, and is demonstrably faster, than state-of-the-art
methods based on LU and QR matrix decompositions. More
specifically, they examined the Frobenius inversion method
for a complex matrix $Z = A+i B$:
\begin{equation}
    Z^{-1} = (A + BA^{-1}B)^{-1} - i A^{-1}B(A+BA^{-1} B)^{-1},
    \label{eq:frob}
\end{equation}
which requires only the evaluation of inverses of real 
matrices $A$ and $B$. The expression \eqref{eq:frob} requires only
the $A$ component of nonsingular $Z$ to be nonsingular, but 
it can be reformulated to require only $B$ to be 
nonsingular\footnote{Randomization, e.g., as suggested by Zielinski 
\cite{inversion_ref9}, can reduce assumptions to the minimum of 
nonsingular $Z$, but at the expense of additional overhead. In this 
paper we will focus only on generalizations of \eqref{eq:frob} with
the understanding that our results can be trivially reformulated
to satisfy different assumptions about the structure of $Z$.}.
We now show that a similar approach can be applied to 
obtain improved computational efficiency for the case of
quaternion matrices.
 
\section{Quaternion Matrices}

Space of quaternions, denoted $\mathbb{H}$, generalizes
the complex space $\mathbb{C}$ by introducing two additional
complex components/dimensions, denoted $j$ and $k$. Thus,
whereas a complex number can be represented with two real
scalar parameters $a$ and $b$ as $a+ib$, a quaternion has
four real parameters and has the form $a+ib+jc+kd$, i.e., 
it has one real and three imaginary components. Quaternions
are algebraically important because they form a linear space
with properties analogous to matrices, e.g., they do not
commute, but retain many valuable properties of complex
numbers, e.g., is a normed division algebra.

The imaginary components of a quaternion are analogous
to the complex imaginary component in that
\begin{equation}
i^2 = \text{-1},~ j^2 = \text{-1},~ \mbox{and}~ k^2 = \text{-1}.
\end{equation} 
with component cross products 
\begin{equation}
ij = k,~ jk = i,~ \mbox{and}~ ki = j,
\end{equation} 
and reversing the order of terms multiplies the
result by -$1$, i.e., they anticommute with respect
to reverse cyclical lexicographic order:
\begin{equation}
ji = \text{-}k,~ kj = \text{-}i,~ \mbox{and}~ ik = \text{-}j.
\end{equation} 

While less familiar, practical interest in quaternions has increased
steadily since the 1990s, and are now widely used in computer
graphics and GIS systems to perform rotations around a specified
vector, as well as in automatic control and system analysis
applications involving state evolution on a sphere, e.g., circular
orbits around the earth. More recently, quaternion matrices, i.e.,
matrices with quaternion elements, have increased in prominence
in applications ranging from representations of color information 
for image processing to uses in neural networks.

A quaternion matrix can be represented in the following form:
\begin{equation}
	H = A + iB + jC + kD,
	\label{qmat}
\end{equation}  
for real-valued matrices $A$, $B$, $C$, and $D$. A natural
question, then, is whether the computation of the inverse of a
given quaternion matrix can be made more efficient by using
a decomposition inspired by that of Frobenius for complex
matrices.

\section{Inverting Quaternion Matrices}

Our first observation is that the Frobenius method
can be applied directly in the case that any two of the
matrices $B$, $C$, and $D$ are zero. For example,
if $B$ and $C$ are zero, then $H =A+kD$ and 
\begin{equation}
    H^{-1} = (A + DA^{-1}D)^{-1} - k A^{-1}D(A+DA^{-1} D)^{-1},
    \label{frobj}
\end{equation}
where the inverse is computed in the $j$ complex plane.
From this observation, it can be concluded that any generalized
Frobenius-type inverse must satisfy the above with regard to
each of $i$, $j$, and $k$ when the matrix coefficients on the
other two are zero. 

We observe that there is an inclusion of $\mathbb{R}$-algebras: 
\[
\mathbb{R} \subseteq \mathbb{C} \subseteq \mathbb{H}.
\]
This implies that
\[
M_n(\mathbb{R}) \subseteq M_n(\mathbb{C}) \subseteq M_n(\mathbb{H}).
\]
We remark that $M_n(\mathbb{C})$ is an $M_n(\mathbb{R})$-bimodule, since $M_n(\mathbb{C}) \simeq M_n(\mathbb{R}) \otimes_{\mathbb{R}} \mathbb{C}$. Moreover, $\mathbb{C}$ is a quadratic field extension of $\mathbb{R}$, over which the matrix inversion is thoroughly discussed in \cite{DLY22}. As a comparison, we notice that $\mathbb{H}$ is not an algebra over $\mathbb{C}$, according to the Gelfand-Mazur theorem \cite{Gelfand41,Mallios86}. Thus the method developed in \cite{DLY22} does not apply directly to the inversion in $M_n(\mathbb{H})$ over $M_n(\mathbb{C})$. However, we still have 
\[
M_n(\mathbb{H}) \simeq M_n(\mathbb{C}) \otimes_{\mathbb{C}} \mathbb{H}.
\]
This implies that given any $Z, W \in M_n(\mathbb{H})$, we may write 
\begin{align*}
Z &= A + iB + jC + kD = (A + iB) + j(C-iD),\\ W  &= E + iF + jG + kH = (E + iF) + j(G-iH),
\end{align*}
where $A,B,C,D,E,F,G,H\in M_n(\mathbb{R})$. We observe that 
\begin{align*}
WZ &=\left[ (A + iB) + j(C-iD) \right] \left[ (E + iF) + j(G-iH) \right] \\
&= \left[(A+iB)(E + iF) - (C+iD)(G-iH)\right] +j \left[ (C-iD)(E+iF) + (A - iB)(G - iH) \right],
\end{align*}
Thus $W$ is the inverse of $Z$ if and only if 
\begin{align}
(A+iB)(E + iF) - (C+iD)(G-iH) &= I_n, \label{eq:equation for inversion-1}\\
(C-iD)(E+iF) + (A - iB)(G - iH) &= 0. \label{eq:equation for inversion-2}
\end{align}
Solving \eqref{eq:equation for inversion-1} and \eqref{eq:equation for inversion-2} for $E,F,G$ and $H$, we obtain the first two items of the lemma that follows. 
\begin{lemma}\label{lem:inversion formula}
Let $Z = A + iB + jC + kD\in M_n(\mathbb{H})$. 
\begin{enumerate}[(a)]
\item \label{lem:inversion formula-1} if $A + iB$ is invertible in $M_n(\mathbb{C})$, then $(A+iB) + (C+iD)(A-iB)^{-1}(C-iD)$ is also invertible and 
\begin{equation}
\begin{split}
Z^{-1} &=\left[ (A+iB) + (C+iD)(A-iB)^{-1}(C-iD)\right]^{-1} \\
&-j (A - iB)^{-1}(C-iD)\left[ (A+iB) + (C+iD)(A-iB)^{-1}(C-iD)\right]^{-1}.
\end{split}
\end{equation}
\item \label{lem:inversion formula-2}if $C + iD$  invertible in $M_n(\mathbb{C})$, then $(A+iB)(C-iD)^{-1}(A-iB) - (C + iD)$ is invertible and 
\begin{equation}
\begin{split}
Z^{-1} &= (C - iD)^{-1}(A-iB)\left[ (A+iB)(C-iD)^{-1}(A-iB) + (C + iD)\right]^{-1} \\
&+j  \left[ - (A+iB)(C-iD)^{-1}(A-iB) - (C + iD)\right]^{-1}.
\end{split}
\end{equation}
\end{enumerate}
\end{lemma}
\begin{proof}
It is sufficient to prove \eqref{lem:inversion formula-1}. To that end, we recall that a quaternion $z = a + ib + jc + kd \in \mathbb{H}$ can be represented as a $2\times 2$ complex matrix $\begin{bmatrix}
a + ib & c + id \\
-c + id & a - ib
\end{bmatrix}$, where $a,b,c,d\in \mathbb{R}$. Therefore, we may also represent $Z= A + iB + jC + kD \in M_n(\mathbb{H})$ as a complex $2\times 2$ block matrix 
\[
\widehat{Z} = \begin{bmatrix}
A + iB & C + iD \\
-C + iD & A - iB
\end{bmatrix} \in M_{2n}(\mathbb{C}),
\]
where $A,B,C,D\in M_n(\mathbb{R})$. As a consequence, inverting $Z$ is equivalent to inverting $\widehat{Z}$. If $A - iB$ is invertible, then we may invert $\widehat{Z}$ by the Schur complement, i.e., 
\[
\widehat{Z}^{-1} = \begin{bmatrix}
W_1 &  -\overline{W_2}\\
W_2  & \overline{W_1}
\end{bmatrix},
\]
where $\overline{X}$ denotes the conjugate of a matrix $X$ and 
\begin{align*}
W_1 &= \left[ (A+ iB) + (C+iD)(A-iB)^{-1}(C-iD) \right]^{-1}, \\
W_2 &= (A-iB)^{-1}(C - iD) \left[ (A + iB) + (C + iD)(A- iB)^{-1}(C - iD) \right]^{-1}.
\end{align*}
Here the invertibility of $(A+ iB) + (C+iD)(A-iB)^{-1}(C-iD)$ is ensured by that of $A+iB$ and equations  \eqref{eq:equation for inversion-1} and \eqref{eq:equation for inversion-2}. Hence the inverse of $Z$ is 
\[
Z^{-1} = W_1 - j W_2
\]
and this completes the proof.
\end{proof}
Clearly, one can iterate the roles of $A,B,C,D$ in Lemma~\ref{lem:inversion formula} to obtain more inversion formulae in $M_n(\mathbb{H})$.
\subsection{The Frobenius algorithm over \texorpdfstring{$M_n(\mathbb{C})$}{} and its optimality}
 For notational simplicity, we define 
\small \begin{align*}
\mathcal{E}_1 &\coloneqq \lbrace
A + i B + j C + kD\in M_n(\mathbb{H}): A + iB \in \GL_n(\mathbb{C}) \rbrace, \\
\mathcal{E}_2 &\coloneqq \lbrace
A + i B + j C + kD\in M_n(\mathbb{H}): C + iD \in \GL_n(\mathbb{C}) \rbrace.
\end{align*} \normalsize
We describe formulae in Lemma~\ref{lem:inversion formula} \eqref{lem:inversion formula-1}--\eqref{lem:inversion formula-2} as a pseudocode in Algorithm~\ref{alg:inverse in MnH}.  
\begin{algorithm}[!htbp]
\caption{complex Frobenius inversion}
\label{alg:inverse in MnH}
\begin{algorithmic}[1]
\renewcommand{\algorithmicrequire}{ \textbf{Input}}
\Require
$Z = A + i B + j C + k D \in \mathcal{E}_1 \cup \mathcal{E}_2$
\renewcommand{\algorithmicensure}{ \textbf{Output}}
\Ensure
inverse of $Z$
\If {$Z \in \mathcal{E}_1$}
\State compute $X_1 = (A - iB)^{-1}$, $X_2 = X_1(C - iD)$, $X_3 = (C+iD)X_2$, 
\State compute $X_4 =( (A + iB) + X_3 )^{-1}$, $X_5 = -X_2 X_4$.
\State\Return $Z^{-1} = X_4 + j X_5$.
\ElsIf {$Z \in \mathcal{E}_2$}
\State compute $X_1 = (C-iD)^{-1}$, $X_2 = X_1 (A-iB)$, $X_3 = (A+iB)X_2$
\State compute $X_4 =( X_3 + (C + iD) )^{-1}$, $X_5 = X_2 X_4$
\State\Return $Z^{-1} = X_5 - j X_4$.
\EndIf 
\end{algorithmic}
\end{algorithm}
It is clear that Algorithm~\ref{alg:inverse in MnH} costs $2$ inversions and $3$ multiplications in $M_n(\mathbb{C})$. Next we analyze the optimlity of Algorithm~\ref{alg:inverse in MnH}. To that end, we need the following two lemmas.
\begin{lemma}\cite{DLY22}
For any field $\Bbbk$ and $A,B\in M_n(\Bbbk)$ such that both $A$ and $A+BA^{-1}B$ are invertible, it requires at least two matrix inversions to compute $(A+BA^{-1}B)^{-1}$, no matter how many matrix additions, matrix multiplications or scalar multiplications in $\Bbbk$ are used.
\label{lemma 3-3}
\end{lemma}

\begin{lemma}
Let $p_1,p_2,q_1,q_2$ be polynomials over $\mathbb{C}$ in $n$ variables. Assume that 
\[
p_1(a) q_2(a) = p_2(a) q_1(a)
\]
for any $a = (a_{1},\dots,a_{n})\in\mathbb{R}^{n}$ such that $q_1(a) \ne 0$ and $q_2(a) \ne 0$. Then $p_{1} q_{2} = p_2 q_{1}$.
\label{lemma 3-4}
\end{lemma}
\begin{proof}
We observe that any polynomial $f$ over $\mathbb{C}$ can be uniquely written as $f= \operatorname{Re}(f)+ i \im(f)$, where $\operatorname{Re}(f)$ and $\im(f)$ are polynomials with real coefficients. We assume that $p_{1}q_{2}-p_{2}q_{1}\not=0$. Then it is clear that either $\operatorname{Re}(p_{1}q_{2}-p_{2}q_{1})\not=0$ or $\im(p_{1}q_{2}-p_{2}q_{1})\not=0$. Without loss of generality, we may assume that $\operatorname{Re}(p_{1}q_{2}-p_{2}q_{1})\not =0$. Similarly we may also assume $\operatorname{Re}(q_{1}) \ne 0$ and $\operatorname{Re}(q_{2}) \ne 0$. Since $\mathbb{R}$ is an infinite field, there exists $a = (a_{1},\cdots,a_{n})\in \mathbb{R}^{n}$, such that 
\[
\left( \operatorname{Re}(p_{1}q_{2}-p_{2}q_{1})\operatorname{Re}(q_{1})\operatorname{Re}(q_{2}) \right)(a)\not=0.
\] 
Hence $(p_{1}/q_{1}-p_{2}/q_{2})(a)\not =0$, $q_1(a) \ne 0$, $q_2(a)\not=0$, which contradicts to the assumption.
\end{proof}

Now we are ready to prove the optimality of Algorithm~\ref{alg:inverse in MnH}. To simplify the notation, we define 
\[
Z_1 = A + iB,\quad Z_2 = C + iD.
\]
Thus $Z = A + iB + jC + kD = Z_1 + j Z_2$. According to Algorithm~\ref{alg:inverse in MnH}, we may also assume that $Z_1$ is invertible so that $Z^{-1} = W_1 - j W_2$ where 
\begin{align*}
W_1 = \left( Z_1 + Z_2 \overline{Z_1}^{-1}\overline{Z_2} \right)^{-1},W_2 = \overline{Z_1}^{-1}\overline{Z_2} \left( Z_1 + Z_2\overline{Z_1}^{-1}\overline{Z_2} \right)^{-1}.
\end{align*}
Here $\overline{X}$  denotes the complex conjugate of a complex matrix $X$.
\begin{theorem}\label{thm:optimality}
It requires at least 1 addition and 2 inversions to compute $\left( Z_1 + Z_2\overline{Z_1}^{-1}\overline{Z_2} \right)^{-1}$. In particular, Algorithm~\ref{alg:inverse in MnH} is optimal in the sense of the least number of complex matrix inversions.
\end{theorem}
We remark that Theorem~\ref{thm:optimality} should be understood as follows: $2$ is the minimum number of complex matrix inversions required to compute $Z^{-1}$, regardless of the number of additions, matrix multiplications, scalar multiplications and conjugations used.
\begin{proof}
1 addition is obviously necessary, thus it is sufficient to prove the necessity of 2 inversions. We proceed by contradiction. Assume that there exists an algorithm which costs only one inversion. Then we may find complex noncommutative polynomials $f$ and $g$ such that 
\begin{equation}\label{thm:optimality:eq1}
\left( Z_1 + Z_2\overline{Z_1}^{-1}\overline{Z_2} \right)^{-1} = g(Z_1,\overline{Z_1},Z_2,\overline{Z_2},f(Z_1,Z_2,\overline{Z_1},\overline{Z_2})^{-1},\overline{f(Z_1,Z_2,\overline{Z_1},\overline{Z_2})^{-1}})
\end{equation}
Clearly, one can further find complex noncommutative polynomials $h$ and $\phi$ so that \eqref{thm:optimality:eq1} may be simplified as follows.
\begin{align}
\left( Z_1 + Z_2\overline{Z_1}^{-1}\overline{Z_2} \right)^{-1} &= g\left( Z_1,\overline{Z_1},Z_2,\overline{Z_2},f(Z_1,Z_2,\overline{Z_1},\overline{Z_2})^{-1},\overline{f(Z_1,Z_2,\overline{Z_1},\overline{Z_2})^{-1}} \right) \nonumber \\
&= h\left( Z_1,\overline{Z_1},Z_2,\overline{Z_2},\left( f(Z_1,Z_2,\overline{Z_1},\overline{Z_2})\overline{f(Z_1,Z_2,\overline{Z_1},\overline{Z_2})} \right)^{-1} \right) \nonumber \\
&= h\left( Z_1,\overline{Z_1},Z_2,\overline{Z_2},\phi(Z_1,Z_2,\overline{Z_1},\overline{Z_2})^{-1} \right). \label{thm:optimality:eq2}
\end{align}
Next we consider the rational map $\Phi:M_n(\mathbb{C}) \times M_n(\mathbb{C}) \dashrightarrow M_n(\mathbb{C})$ defined by
\[
\Phi(Z_1,Z_2) = \left( Z_1 + Z_2 Z_1^{-1}Z_2 \right)^{-1}  - h\left( Z_1,Z_1,Z_2,Z_2,\phi(Z_1,Z_2,Z_1,Z_2)^{-1} \right).
\]
In particular, elements of the matrix $\Phi(Z_1,Z_2)$ are rational functions on $M_n(\mathbb{C}) \times M_n(\mathbb{C})$. By \eqref{thm:optimality:eq2}, we see immediately that $\Phi \equiv 0$ on $M_n(\mathbb{R}) \times M_n(\mathbb{R})$ whenever it is defined. Therefore Lemma~\ref{lemma 3-4} implies that elements of $\Phi$ must vanish on $M_n(\mathbb{C}) \times M_n(\mathbb{C})$ whenever they are defined. In particular, $\Phi(Z_1,Z_2) \equiv 0$ on $M_n(\mathbb{C}) \times M_n(\mathbb{C})$ whenever it is defined, from which we may conclude that $\left( Z_1 + Z_2 Z_1^{-1} Z_2 \right)^{-1}$ can be computed by one complex matrix inversion. This leads to a contradiction to Lemma~\ref{lemma 3-3}.
\end{proof}
\subsection{The Frobenius algorithm over \texorpdfstring{$M_n(\mathbb{R})$}{} and its complexity}
According to \cite{DLY22}, an inversion for a generic element in $M_n(\mathbb{C})$ costs $2$ inversions and $3$ multiplications in $M_n(\mathbb{R})$, and a multiplication in $M_n(\mathbb{C})$ costs $3$ multiplications in $M_n(\mathbb{R})$. Therefore, the total cost of inverting a generic $Z\in M_n(\mathbb{H})$ by Algorithm~\ref{alg:inverse in MnH}, together with algorithms in \cite{DLY22} for multiplication and inversion in $M_n(\mathbb{C})$, is at most $4$ inversions and $15$ multiplications in $M_n(\mathbb{R})$. However, we observe that there are actually two repeated multiplications, thus we may obtain Algorithm~\ref{alg:real Frobenius} and the theorem that follows. 
\begin{algorithm}[!htbp]
\caption{real Frobenius inversion}
\label{alg:real Frobenius}
\begin{algorithmic}[1]
\renewcommand{\algorithmicrequire}{ \textbf{Input}}
\Require
generic $Z = A + i B + j C + k D \in M_n(\mathbb{H})$ 
\renewcommand{\algorithmicensure}{ \textbf{Output}}
\Ensure
inverse of $Z$
\State compute $K = C + D$.
\State compute $U_1 = A^{-1}B$, $U_2 = (A + BU_1)^{-1}$, $U_3 = U_1 U_2$, $U_4 = U_3 C$, $U_5 = U_2D$.
\State compute $V_1 = (U_2 + U_3)K - U_4 - U_5$,  
        $V_2 = U_4 - U_5$,
        $V_3 = V_1 - V_2$, $ V_4 = C V_2$, $V_5 = D V_1$.
\State compute
        $W_1 = K  V_3 + A  + V_4 - V_5$,
        $W_2 = B + V_4 + V_5$,
        $W_3 = W_1^{-1}W_2$.
\State compute $ E = (W_1 + W_2 W_3)^{-1}$,
        $E_1 = V_2 E$.
\State compute 
        $F = -W_3 E$,
        $F_1 = V_1 F$.
\State compute 
        $G =  F_1 - E_1 - V_3 (E + F)$,
        $H = F_1 + E_1$.
\State\Return $Z^{-1} = E + iF + jG + k H$.
\end{algorithmic}
\end{algorithm}
\begin{proposition}
Algorithm~\ref{alg:real Frobenius} computes $Z^{-1}$ by $4$ inversions and $13$ multiplications in $M_n(\mathbb{R})$. 
\end{proposition}
\section{Comparison with existing methods}
In this section, we compare the computational overhead of the Generalized Frobenius Inversion (Algorithms~\ref{alg:inverse in MnH} and \ref{alg:real Frobenius}) with existing methods for quaternion matrix inversion. 
\subsection{Four Alternative Methods}
As far as we are aware, there have been only four methods for quaternion matrix inversion previously discussed in the literature. Before we proceed further, we briefly summarize these methods. 
\subsubsection{Complex Method}
Let $\mathcal{\varphi}_{1}: M_n(\mathbb{H}) \to M_{2n}(\mathbb{C})$ be the injective map defined by
\[
\mathcal{\varphi}_{1}(A + iB + jC + kD) = \begin{bmatrix}
A + iB & C + iD \\
-C + iD & A - iB
\end{bmatrix},\quad A,B,C,D\in M_n(\mathbb{R}).
\]
Clearly $\mathcal{\varphi}_{1}$ is a homomorphism of algebras over $\mathbb{R}$ which sends a quaternion matrix into a double sized complex matrix. Converting quaternion matrices into complex matrices is a standard technique \cite{Fuzhen1997,WLZZ2018}, from which we may obtain Algorithm~\ref{alg:inverse in MnH via MnC}.
\begin{algorithm}[!htbp]
\caption{Complex Method}
\label{alg:inverse in MnH via MnC}
\begin{algorithmic}[1]
\renewcommand{\algorithmicrequire}{ \textbf{Input}}
\Require
$Z = A + i B + j C + k D \in M_n(\mathbb{H})$
\renewcommand{\algorithmicensure}{ \textbf{Output}}
\Ensure
inverse of $Z$
\State compute $Y=\mathcal{\varphi}_{1}(Z)$.
\State compute $V=Y^{-1}$.
\State\Return $Z^{-1}=\mathcal{\varphi}_{1}^{-1}(V)$.
\end{algorithmic}
\end{algorithm}
\subsubsection{Real Method}
We may also convert a quaternion matrix into a quadruple sized real matrix \cite{Fuzhen1997,WLZZ2018} via an injective map $\varphi_2: M_n(\mathbb{H}) \to M_{4n}(\mathbb{R})$ defined by 
\[
\mathcal{\varphi}_{2}(A + iB + jC + kD) = \begin{bmatrix}
A & B & C & D \\
-B & A & -D & C \\
-C & D & A & -B \\
-D & -C & B & A
\end{bmatrix},\quad A,B,C,D\in M_n(\mathbb{R}).
\]
Again,  $\varphi_2$ is an $\mathbb{R}$-algebra homomorphism which leads to Algorithm~\ref{alg:inverse in MnH via MnR}.
\begin{algorithm}[!htbp]
\caption{Real Method}
\label{alg:inverse in MnH via MnR}
\begin{algorithmic}[1]
\renewcommand{\algorithmicrequire}{ \textbf{Input}}
\Require
$Z = A + i B + j C + k D \in M_n(\mathbb{H})$
\renewcommand{\algorithmicensure}{ \textbf{Output}}
\Ensure
inverse of $Z$
\State compute $Y=\mathcal{\varphi}_{2}(Z)$.
\State compute $V=Y^{-1}$.
\State\Return $Z^{-1}=\mathcal{\varphi}_{2}^{-1}(V)$.
\end{algorithmic}
\end{algorithm}
\subsubsection{Skew Real Method}
By exploring the block skew-symmetric structure of $\mathcal{\varphi}_{2}(A + iB + jC + kD)$, a more efficient approach is proposed in \cite{Turner2009}. For ease of reference, we record the method in Algorithm~\ref{alg:skew real method}. We remark that the orginal approach in \cite{Turner2009} is presented in terms of block matrices, while Algorithm~\ref{alg:skew real method} is a simplified and more efficient version.
\begin{algorithm}[!htbp]
\caption{Skew Real Method}
\label{alg:skew real method}
\begin{algorithmic}[1]
\renewcommand{\algorithmicrequire}{ \textbf{Input}}
\Require
$Z = A + i B + j C + k D \in M_n(\mathbb{H})$
\renewcommand{\algorithmicensure}{ \textbf{Output}}
\Ensure
inverse of $Z$
\State compute $U_1 = A^{-1}B$, $U_2 = A + BU_1$, $U_3 = U_2^{-1}$, $U_4 = U_1 U_3$.
\State compute $V_1 = U_3 C + U_4 D$, $V_2 = -U_4 C+U_3 D$.
\State compute $W_1 = A + C V_1 - DV_2, W_2 = B + DV_1+ CV_2$, $W_3 = W_1^{-1} W_2$.
\State compute $X = (W_1 + W_2 W_3)^{-1}$, $Y = - W_3 X$, $Z = - V_1 X - V_2 Y$, $W = V_2 X - V_1 Y$.
\State\Return $Z^{-1}=X + iY + jZ + kW$.
\end{algorithmic}
\end{algorithm}
\subsubsection{Quaternion Toolbox For Matlab (QTFM)}. A quaternion inversion algorithm for quaternion matrix inversion is implemented in the QTFM \cite{SB2022}. The main idea behind this approach is to partition a matrix $Z\in M_{n}(\mathbb{H})$ into 
\[
Z = \begin{bmatrix}
Z_{11} & Z_{12} \\
Z_{21} & Z_{22}
\end{bmatrix}
\]
for some $Z_{11}\in M_{\lfloor n/2 \rfloor}(\mathbb{H})$, $Z_{12}\in \mathbb{H}^{\lfloor n/2 \rfloor \times \lceil n/2 \rceil}$, $Z_{21}\in \mathbb{H}^{\lceil n/2 \rceil \times \lfloor n/2 \rfloor}$ and $Z_{22}\in M_{\lceil n/2 \rceil}(\mathbb{H})$. Then the inverse of $Z$ can be computed by the Schur complement \cite{HJ2012} and one can repeat the procedure to obtain Algorithm~\ref{alg:toolbox} for $Z^{-1}$. The key feature of Algorithm~\ref{alg:toolbox} that distinguishes it from Algorithms~\ref{alg:inverse in MnH}--\ref{alg:skew real method} is that operations in Algorithm~\ref{alg:toolbox} are all over $\mathbb{H}$, rather than $\mathbb{C}$ or $\mathbb{R}$. 
\begin{algorithm}[!htbp]
\caption{toolbox}
\label{alg:toolbox}
\begin{algorithmic}[1]
\renewcommand{\algorithmicrequire}{ \textbf{Input}}
\Require
$Z \in M_n(\mathbb{H})$
\renewcommand{\algorithmicensure}{ \textbf{Output}}
\Ensure
inverse of $Z$
\State partition
$Z = \begin{bmatrix}
Z_{11} & Z_{12} \\
Z_{21} & Z_{22}
\end{bmatrix}$. 
\State compute $U_1 = Z_{11}^{-1}$, % IA
$U_2 =U_1 Z_{12}$, % IAB
$U_3 = Z_{21} U_1$, % CIA
\Comment{inversion in $M_{\lfloor n/2 \rfloor}(\mathbb{H})$}
\State compute $U_4 = (Z_{22} - U_3 Z_{12} )^{-1}$,%T
$U_5 = U_4 U_3$.  % TCIA 
\Comment{inversion in $M_{\lceil n/2 \rceil}(\mathbb{H})$}
\State compute $R = \begin{bmatrix}
U_1 + U_2 U_5 & -U_2 U_4 \\
- U_5 & U_4
\end{bmatrix}$
\State\Return $Z^{-1}=X + iY + jZ + kW$.
\end{algorithmic}
\end{algorithm}
\subsection{Comparison}
In this subsection, we compare computational complexities of Algorithms~\ref{alg:inverse in MnH}--\ref{alg:toolbox}. We denote by $\operatorname{inv}_{n,\mathbb{F}}$ (resp. $\operatorname{mult}_{n,\mathbb{F}}$, $\operatorname{add}_{n,\mathbb{F}}$) the inversion (resp. multiplication, addition) in $M_n(\mathbb{F})$, where $\mathbb{F} = \mathbb{R},\mathbb{C}$ or $\mathbb{H}$. In Table~\ref{tab:complexity comparison}, we record operations required by each of Algorithms~\ref{alg:inverse in MnH}--\ref{alg:toolbox} in the column labelled by ``operations".

Let $\mathcal{A}$ be an algorithm for matrix multiplication over $\mathbb{H}$. We denote by $c_{n,\mathbb{H}}$, $c_{n,\mathbb{C}}$ and $c_{n,\mathbb{R}}$ the complexity of $\mathcal{A}$ restricted to $M_n(\mathbb{H})$, $M_n(\mathbb{C})$ and $M_n(\mathbb{R})$ respectively. 
%Furthermore, we let $\tau \in (0,1]$ be a constant number such that the product of two $n\times n$ matrices over $\mathbb{F}$ can be computed by $\tau c_{n,\mathbb{F}}$ flops if one of them is upper or lower triangular. Here $\mathbb{F}$ denotes $\mathbb{H},\mathbb{C}$ or $\mathbb{R}$.
\begin{theorem}
If we compute $\operatorname{mult}_{n,\mathbb{H}}$ by $\mathcal{A}$ and compute $\operatorname{inv}_{n,\mathbb{C}}$ by the LU-decomposition method, then the complexity (ignoring the lower order terms) of Algorithms~\ref{alg:inverse in MnH}--\ref{alg:toolbox} is as shown in the column of Table~\ref{tab:complexity comparison} with label ``complexity". In particular, if $\mathcal{A}$ is the usual algorithm\footnote{This refers to the algorithm obtained by the definition of matrix multiplication.} for matrix multiplication, and if operations in $\mathbb{H}$ and $\mathbb{C}$ are computed by usual methods\footnote{This refers to algorithms obtained by definitions of multiplication and addition in $\mathbb{H}$ and $\mathbb{C}$.} over $\mathbb{R}$, then among Algorithms~\ref{alg:inverse in MnH}--\ref{alg:toolbox}, Algorithm~\ref{alg:real Frobenius} is optimal.
\end{theorem}
\begin{proof}
We recall that inverting an $n\times n$ matrix by the LU-decomposition over $\mathbb{C}$ (resp. $\mathbb{R}$) costs (ignoring lower order terms) $\frac{4n^3}{3} (\operatorname{mult}_{1,\mathbb{C}} + \operatorname{add}_{1,\mathbb{C}})$ (resp. $\frac{4n^3}{3} (\operatorname{mult}_{1,\mathbb{R}} + \operatorname{add}_{1,\mathbb{R}})$). If $\mathcal{A}$ is the usual algorithm for matrix multiplication, then we have 
\[
c_{n,\mathbb{F}} = n^3 (\operatorname{mult}_{1,\mathbb{F}} + \operatorname{add}_{1,\mathbb{F}}),
\] 
where $\mathbb{F} = \mathbb{H},\mathbb{C}$ or $\mathbb{R}$. Moreover, if we compute operations in $\mathbb{H}$ and $\mathbb{C}$ in the usual way, then we also have 
\begin{align*}
\operatorname{mult}_{1,\mathbb{H}} &= 16 \operatorname{mult}_{1,\mathbb{R}} + 12 \operatorname{add}_{1,\mathbb{R}},\quad \operatorname{add}_{1,\mathbb{H}} =  4 \operatorname{add}_{1,\mathbb{R}}, \\
\operatorname{mult}_{1,\mathbb{C}} &= 4 \operatorname{mult}_{1,\mathbb{R}} + 2 \operatorname{add}_{1,\mathbb{R}},\quad \operatorname{add}_{1,\mathbb{C}} =  2 \operatorname{add}_{1,\mathbb{R}}.
\end{align*}
This implies that Algorithms~\ref{alg:inverse in MnH}--\ref{alg:toolbox} respectively cost $\frac{136 n^3}{3}, \frac{110 n^3}{3}, \frac{256 n^3}{3}, \frac{512 n^3}{3}, \frac{128 n^3}{3}, 128 n^3$ flops in $\mathbb{R}$, from which we may conclude the optimality of Algorithm~\ref{alg:real Frobenius}.
\end{proof}
\begin{table}[htbp]
\begin{center}
 \begin{tabular}{p{60pt}p{220pt}p{160pt}}
\toprule  Algorithm  & operations  & complexity\\
\midrule 
\ref{alg:inverse in MnH} & $2 \operatorname{inv}_{n,\mathbb{C}} + 3 \operatorname{mult}_{n,\mathbb{C}} + \operatorname{add}_{n,\mathbb{C}}$  & $\frac{8n^3}{3} (\operatorname{mult}_{1,\mathbb{C}} + \operatorname{add}_{1,\mathbb{C}}) + 3 c_{n,\mathbb{C}}$ \\
\ref{alg:real Frobenius} & $4 \operatorname{inv}_{n,\mathbb{R}} + 13 \operatorname{mult}_{n,\mathbb{R}} + 17 \operatorname{add}_{n,\mathbb{R}}$  & $\frac{16n^3}{3} (\operatorname{mult}_{1,\mathbb{R}} + \operatorname{add}_{1,\mathbb{R}}) + 13 c_{n,\mathbb{R}}$\\
\ref{alg:inverse in MnH via MnC} & $\operatorname{inv}_{2n,\mathbb{C}}$ & $\frac{32n^3}{3} (\operatorname{mult}_{1,\mathbb{C}} + \operatorname{add}_{1,\mathbb{C}})$ \\
\ref{alg:inverse in MnH via MnR}  & $\operatorname{inv}_{4n,\mathbb{R}}$   & $\frac{256n^3}{3} (\operatorname{mult}_{1,\mathbb{R}} + \operatorname{add}_{1,\mathbb{R}})$\\
\ref{alg:skew real method} &  $4 \operatorname{inv}_{n,\mathbb{R}} + 16 \operatorname{mult}_{n,\mathbb{R}} + 10 \operatorname{add}_{n,\mathbb{R}}$ & $\frac{16n^3}{3} (\operatorname{mult}_{1,\mathbb{R}} + \operatorname{add}_{1,\mathbb{R}}) + 16 c_{n,\mathbb{R}}$ \\
\ref{alg:toolbox} & $n \operatorname{inv}_{1,\mathbb{H}} + \sum_{j=1}^{\log_2n} 2^{j} (3 \operatorname{mult}_{n/2^j,\mathbb{H}} + \operatorname{add}_{n/2^j,\mathbb{H}})  $ & $3 \sum_{j=1}^{\log_2n} 2^{j} c_{n/2^j,\mathbb{H}}$ \\
\bottomrule
\end{tabular}
\end{center}
\caption{comparison of complexities}
\label{tab:complexity comparison}
\end{table}
\begin{corollary}
Let $\mathcal{A}$ be an algorithm for matrix multiplication that is faster than the usual algorithm. If we compute operations in $\mathbb{H}$ and $\mathbb{C}$ by usual methods over $\mathbb{R}$, then Algorithm~\ref{alg:real Frobenius} is still optimal among Algorithms~\ref{alg:inverse in MnH}--\ref{alg:toolbox}.
\end{corollary}
We conclude this subsection by a remark on the computation of operations in $\mathbb{H}$ and $\mathbb{C}$ over $\mathbb{R}$. It is easy to prove that the usual algorithms for $\operatorname{add}_{1,\mathbb{H}}$ and $\operatorname{add}_{1,\mathbb{C}}$ are already optimal. It is known \cite{HL75,Hans75} that any non-commutative algorithm for $\operatorname{mult}_{1,\mathbb{H}}$ costs at least $8$ multiplications in $\mathbb{R}$, yet there is no known algorithm for $\operatorname{mult}_{1,\mathbb{H}}$ costs $8$ multiplications and less than $32$ additions in $\mathbb{R}$. Thus the usual algorithm is in fact the optimal one (in the sense of total cost of multiplications and additions ) among all existing\footnote{We reiterate that we are assessing optimality in terms of the set of cubic-complexity quaternion algorithms currently available. Those algorithms, and ours, can of course be applied as the base case within the generic recursions used by the best available Strassen-type algorithm, with the optimal among that set clearly providing the best coefficient on the overall subcubic complexity.} algorithms for $\operatorname{mult}_{1,\mathbb{H}}$. 

Similarly, it is well-known \cite{STRASSEN1983645,KL2018, STRASSEN+1987+406+443,Ramachandra1973VermeidungVD,algebraiccomplexity} that any non-commutative algorithm for $\operatorname{mult}_{1,\mathbb{C}}$ requires at least $3$ multiplications in $\mathbb{R}$ and the optimality is achieved by the Gauss algorithm. However, the Gauss algorithm costs $5$ additions in $\mathbb{R}$. Therefore, the usual algorithm is most efficient among all known algorithms for $\operatorname{mult}_{1,\mathbb{C}}$.
\subsection{Experiments}
For each integer $1 \le m \le 50$, we take $n = 100m$ and apply Algorithms~\ref{alg:inverse in MnH}--\ref{alg:toolbox} to invert random matrices in $M_{n}(\mathbb{H})$. To be more precise, we generate $n\times n$ real matrices $A,B,C$ and $D$ whose elements are uniformly drawn from the interval $(-1,1)$. Then with probability one, Algorithms~\ref{alg:inverse in MnH}--\ref{alg:toolbox} are all applicable to $Z = A + i B + j C + k D\in M_n(\mathbb{H})$. For each $m$, we repeat the above procedure $50$ times to obtain $50$ samples for our test.

We record in Figure~\ref{fig:time} the average running time of each of Algorithms~\ref{alg:inverse in MnH}--\ref{alg:toolbox}. For better comparison, we also compute the ratio 
\[
r_{n,s} = \frac{t_{n,5}}{t_{n,s}},\quad s\in \{1,2,3,4,6\}.
\]  
Here $t_{n,p}$ denotes the average running time of Algorithm~$p$ for instances of size $n\times n$. By definition, $r_{n,s} > r_{n,s'}$ indicates that Algorithm~$s$ is more efficient than Algorithm~$s'$. Figure~\ref{fig:ratio} exhibits how $r_{n,s}$ varies as $n$ grows for each $s\in \{1,2,3,4,6\}$. It is clear from Figures~\ref{fig:time} and \ref{fig:ratio} that Algorithm~\ref{alg:real Frobenius} is significantly faster than all other five algorithms.
\begin{figure*}
\centering
\includegraphics[scale=0.35]{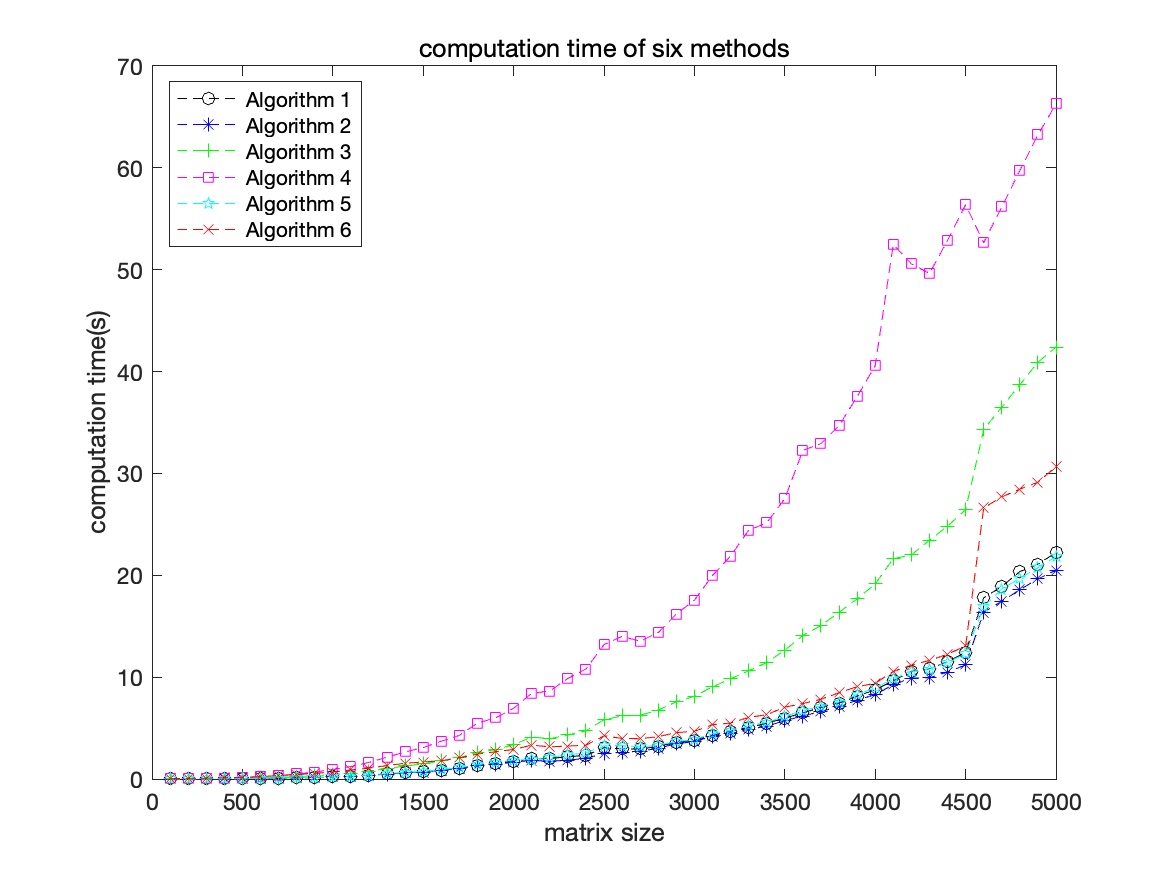}
\caption{Algorithm~\ref{alg:real Frobenius} (blue) incurs significantly less computational overhead compared to other five algorithms.}
\label{fig:time}
\end{figure*}

\begin{figure*}
\centering
\includegraphics[scale=0.35]{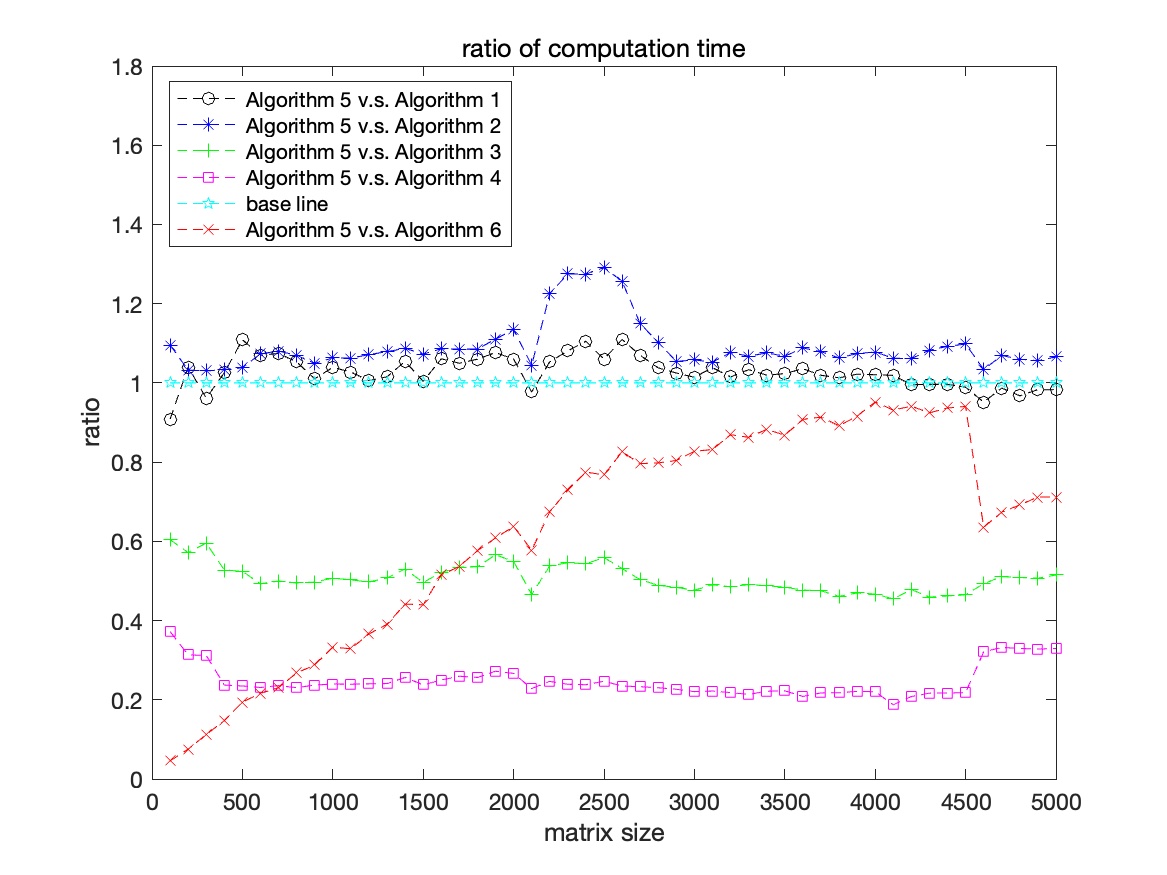}
\caption{The ratio (blue) of Algorithm~\ref{alg:real Frobenius} is significantly greater than ratios of other algorithms.}
\label{fig:ratio}
\end{figure*}

Let $\widehat{Z}$ be the inverse of $Z$ obtained by one of Algorithms~\ref{alg:inverse in MnH}--\ref{alg:toolbox}. We measure the quality of $\widehat{Z}$ by the mean right residual:
\[
\operatorname{res}_{n} = \frac{\lVert Z \widehat{Z} - I_n \rVert_F}{n^2}. 
\]
Here $\lVert \cdot \rVert_F$ is the Frobenius norm on $M_n(\mathbb{H})$. We record in Figure~\ref{fig:error} the average of the mean right residual of $50$ samples for each of Algorithms~\ref{alg:inverse in MnH}--\ref{alg:toolbox}. From Figure~\ref{fig:error}, we may easily conclude that the mean right residual for Algorithms~\ref{alg:real Frobenius}, \ref{alg:skew real method} and \ref{alg:toolbox} is less than $5\cdot 10^{-13}$, while the other three algorithms have relatively smaller mean right residual.
\begin{figure*}
\centering
\includegraphics[scale=0.35]{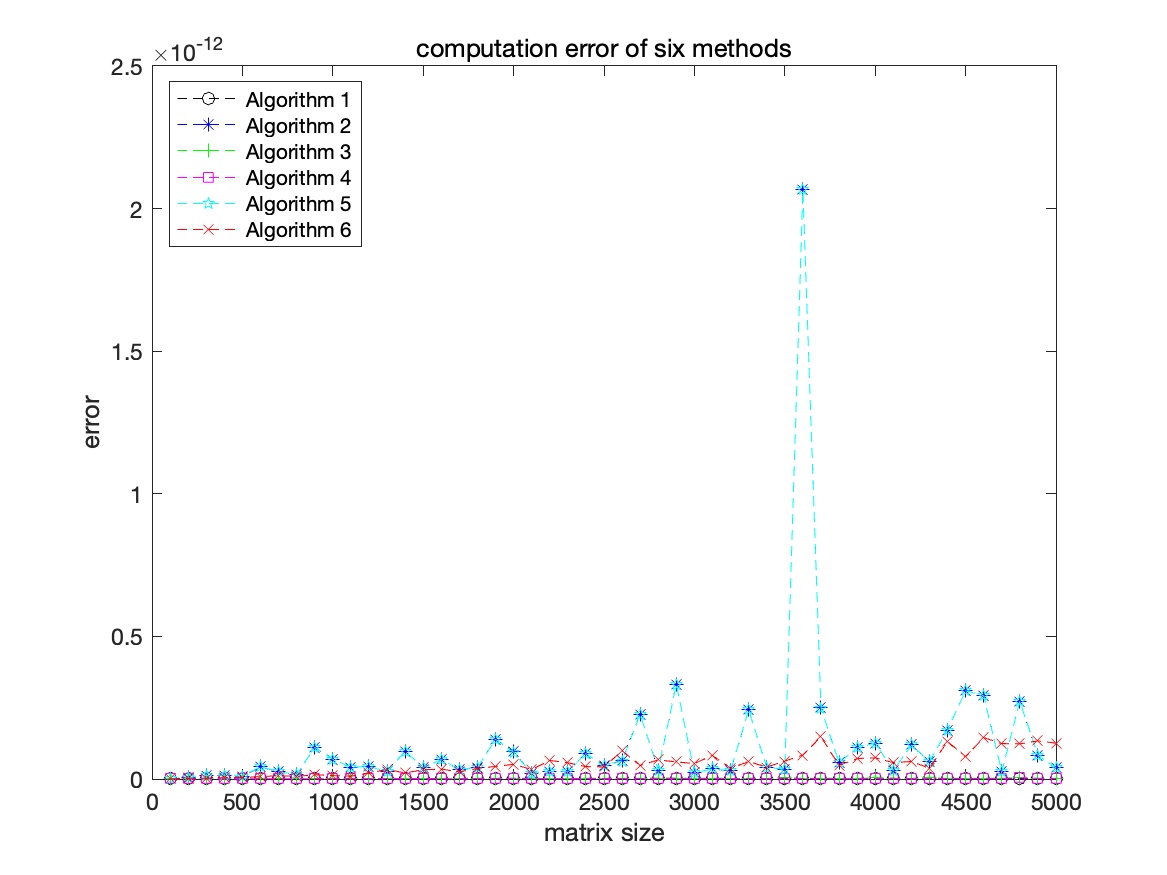}
\caption{The mean right residual of Algorithm~\ref{alg:inverse in MnH} is at most $5\cdot 10^{-13}$ for all but one dimensions.}
\label{fig:error}
\end{figure*}

\section{Discussion}

In this paper we have presented a new algorithm for efficiently inverting quaternion matrices, and we
have provided empirical results showing that it is the fastest among quaternion inversion algorithms
found in the literature. Consequently, it represents the optimizing recursive base case for Strassen-type
inversion algorithms to achieve the best combination of computational complexity and practical
performance for inverting quaternion matrices.

\bibliographystyle{abbrv}
\bibliography{FrobQuat-05_03_2023}

\end{document}